\documentclass{article}%
\usepackage{amsmath}
\usepackage{amsfonts}
\usepackage[colorlinks=true, pdfstartview=FitV, linkcolor=blue, citecolor=red,
urlcolor=blue]{hyperref}
\usepackage{color}
\usepackage{amssymb}
\usepackage{graphicx}%
\providecommand{\U}[1]{\protect\rule{.1in}{.1in}}
\newtheorem{theorem}{Theorem}

\newtheorem{lemma}[theorem]{Lemma}

\newtheorem{remark}[theorem]{Remark}

\newenvironment{proof}[1][Proof]{\noindent\textbf{#1.} }{\ \rule{0.5em}{0.5em}}
\allowdisplaybreaks
\begin{document}

\title{Stieltjes constants appearing in the Laurent expansion of the hyperharmonic
zeta function}
\author{M\"{u}m\"{u}n Can, Ayhan Dil and Levent Karg\i n\\{\small {Department of Mathematics, Akdeniz University, TR-07058 Antalya,
Turkey}}\\{\small {mcan@akdeniz.edu.tr, adil@akdeniz.edu.tr, lkargin@akdeniz.edu.tr,}}}
\maketitle

\begin{abstract}
In this paper, we consider meromorphic extension of the function
\[
\zeta_{h^{\left(  r\right)  }}\left(  s\right)  =\sum_{k=1}^{\infty}%
\frac{h_{k}^{\left(  r\right)  }}{k^{s}},\text{ }\operatorname{Re}\left(
s\right)  >r,
\]
(which we call \textit{hyperharmonic zeta function}) where $h_{n}^{(r)}$ are
the hyperharmonic numbers. We establish certain constants, denoted
$\gamma_{h^{\left(  r\right)  }}\left(  m\right)  $, which naturally occur in
the Laurent expansion of $\zeta_{h^{\left(  r\right)  }}\left(  s\right)  $.
Moreover, we show that the constants $\gamma_{h^{\left(  r\right)  }}\left(
m\right)  $ and integrals involving generalized exponential integral can be
written as a finite combination of some special constants.

\end{abstract}

\section{Introduction}

The Riemann zeta function which is initially defined by the series%
\[
\zeta\left(  s\right)  =\sum_{n=1}^{\infty}\frac{1}{n^{s}},\text{
}\operatorname{Re}\left(  s\right)  >1,
\]
has an analytic continuation to the whole complex $s$ plane except for a
simple pole $s=1$ with residue $1$. It is well known that around this simple
pole, $\zeta\left(  s\right)  $ has the following Laurent expansion
\begin{equation}
\zeta\left(  s\right)  =\frac{1}{s-1}+\sum_{m=0}^{\infty}\left(  -1\right)
^{m}\frac{\gamma\left(  m\right)  }{m!}\left(  s-1\right)  ^{m},\label{ZL}%
\end{equation}
where the coefficients $\gamma\left(  m\right)  $ are called \textit{Stieltjes
constants}. It is shown by various authors that these constants can be
alternatively presented by the limit%
\begin{equation}
\gamma\left(  m\right)  =\lim_{n\rightarrow\infty}\left(  \sum_{k=1}^{n}%
\frac{\ln^{m}k}{k}-\int_{1}^{n}\frac{\ln^{m}x}{x}dx\right)  ,\text{
}m=0,1,2,\ldots.\label{SC}%
\end{equation}
(see for example \cite{Be,BCh,Ha,LiT} and for an extensive literature
information see \cite{Bl2}). The special case $m=0$ is the famous
Euler-Mascheroni constant $\gamma=\gamma\left(  0\right)
=0.577\,215\,664\,9\ldots$. Besides, the constant $\gamma$\ has relation
$\psi\left(  1\right)  =\Gamma^{^{\prime}}\left(  1\right)  =-\gamma$, where%
\[
\Gamma\left(  x\right)  =\int_{0}^{\infty}t^{x-1}e^{-t}dt,\text{ }x>0
\]
is Euler's gamma function and $\psi\left(  x\right)  =\Gamma^{^{\prime}%
}\left(  x\right)  /\Gamma\left(  x\right)  $ is the digamma function (e.g.
\cite{SC}). It is well-known that for a positive integer $n,$ $\psi\left(
n+1\right)  +\gamma=H_{n}=1+1/2+\cdots+1/n,$ where $H_{n}$ are the harmonic numbers.

There is comprehensive literature on deriving series and integral
representations for the Stieltjes constants and their extensions (see for
example \cite{Bl1,Bl2,C,Cof1,Cof2,HK,FB}). These representations usually allow
a more accurate estimation of mentioned constants (see for example
\cite{Ad,Be,Bl2}).

The Dirichlet series associated with harmonic numbers, so called
\textit{harmonic zeta function}, is defined by
\[
\zeta_{H}\left(  s\right)  =\sum_{k=1}^{\infty}\frac{H_{k}}{k^{s}},\text{
}\operatorname{Re}\left(  s\right)  >1.
\]
The harmonic zeta function has been studied by many authors. Euler \cite[pp.
217--264]{E} gives a closed form formula for $\zeta_{H}\left(  s\right)  $ in
terms of the Riemann zeta values when $s\in\mathbb{N}\backslash\left\{
1\right\}  $. Apostol and Vu \cite{AV} and Matsuoka \cite{Ma} show that the
function $\zeta_{H}$ can be continued meromorphically to the whole complex $s$
plane except for the poles $s=1,$ $s=0$ and $s=1-2j,$ $j\in\mathbb{N}$. Later,
Boyadzhiev et al. \cite{BGP} deal with the Laurent expansion of the harmonic
zeta function\textbf{ }%
\begin{equation}
\zeta_{H}\left(  s\right)  =\frac{a_{-1}}{s-b}+a_{0}+O\left(  s-b\right)  ,
\label{zhL}%
\end{equation}
and give explicitly the coefficient $a_{0}$ when $b=0$ and $b=1-2j$,
$j\in\mathbb{N}$. Recently, using the Ramanujan summation method,
Candelpergher and Coppo \cite{CC} record that the \textit{harmonic Stieltjes
constants} $\gamma_{H}\left(  m\right)  $ defined by the Laurent expansion
\[
\zeta_{H}\left(  s\right)  =\frac{1}{\left(  s-1\right)  ^{2}}+\frac{\gamma
}{s-1}+\sum_{m=0}^{\infty}\left(  -1\right)  ^{m}\frac{\gamma_{H}\left(
m\right)  }{m!}\left(  s-1\right)  ^{m},\text{ }0<\left\vert s-1\right\vert
<1
\]
can be presented as%
\[
\gamma_{H}\left(  m\right)  =\lim_{x\rightarrow\infty}\left(  \sum_{n\leq
x}\frac{H_{n}\ln^{m}n}{n}-\frac{\ln^{m+2}x}{m+2}-\gamma\frac{\ln^{m+1}x}%
{m+1}\right)  .
\]
Besides they also present $\gamma_{H}\left(  0\right)  $ explicitly and
rediscover the coefficient $a_{0}$ in (\ref{zhL}).

We now introduce the main object of this study, \textit{the hyperharmonic zeta
function}:
\[
\zeta_{h^{\left(  r\right)  }}\left(  s\right)  =\sum_{k=1}^{\infty}%
\frac{h_{k}^{\left(  r\right)  }}{k^{s}},\text{ }\operatorname{Re}\left(
s\right)  >r,
\]
where $h_{n}^{(r)}$ are the hyperharmonic numbers defined by \cite{CG}%
\[
h_{n}^{(r)}=\sum_{k=1}^{n}h_{k}^{(r-1)}\text{ with\ }h_{n}^{(0)}=\frac{1}%
{n},\text{ }n,r\in\mathbb{N}.
\]
It is clear that $\zeta_{h^{\left(  0\right)  }}\left(  s\right)
=\zeta\left(  s+1\right)  $ and $\zeta_{h^{\left(  1\right)  }}\left(
s\right)  =\zeta_{H}\left(  s\right)  .$ The Dirichlet series $\sum
_{k=1}^{\infty}h_{k}^{\left(  r\right)  }/k^{s}$\ converges absolutely and
represents an analytic function of $s$ for $\operatorname{Re}\left(  s\right)
>r$ since $h_{n}^{(r)}=O\left(  n^{r-1}\ln n\right)  .$ Kamano \cite{KK} has
shown that the function $\zeta_{h^{\left(  r\right)  }}\left(  s\right)  $ can
be continued meromorphically to the whole complex $s$ plane except for the
double poles at $s=1,2,\ldots,r$ and an infinite number of simple poles at
$s=-k,$ $k\in\mathbb{N}\cup\left\{  0\right\}  $. For positive integer values
of $s$, in which case it is called Euler sum of the hyperharmonic numbers
\cite{MD}, it has enjoyed considerable attention in a number of publications
during the last decade. For instance, evaluations of $\zeta_{h^{\left(
r\right)  }}\left(  m\right)  $, $m\in\mathbb{N}$\ and their extensions in
terms of the Riemann zeta values and some other special constants can be found
in \cite{CKDS,DB,KCDC,MD,Xu2}.

The aim of this paper is to determine certain constants which naturally occur
in the\ Laurent expansion of hyperharmonic zeta function.

The organization of the paper is as follows. In Section \ref{sec-mteo} we
prove that the function $\zeta_{h^{\left(  r\right)  }}$ has a meromorphic
continuation (Theorem \ref{mteo}). In Section \ref{sec-mteo2} we present the
Laurent expansion of the hyperharmonic zeta function in the region
$0<\left\vert s-r\right\vert <1$ (Theorem \ref{mteo2}). For appearing
coefficients, denoted $\gamma_{h^{\left(  r\right)  }}\left(  m\right)  ,$ we
obtain the following limit representation
\begin{equation}
\gamma_{h^{\left(  r\right)  }}\left(  m\right)  =\lim_{x\rightarrow\infty
}\left(  \sum_{n\leq x}\frac{h_{n}^{\left(  r\right)  }\log^{m}n}{n^{r}}%
-\frac{1}{\Gamma\left(  r\right)  }\frac{\ln^{m+2}x}{m+2}+\frac{\psi\left(
r\right)  }{\Gamma\left(  r\right)  }\frac{\ln^{m+1}x}{m+1}\right)
,\label{GSC}%
\end{equation}
by modifying the method of Briggs and Buschman \cite{BB}. It is clear that the
coefficients $\gamma_{h^{\left(  r\right)  }}\left(  m\right)  $, which we
call \textit{hyperharmonic Stieltjes constants}, reduce to the Stieltjes
constants $\gamma\left(  m\right)  $ when $r=0$ (with the assumption
$\psi\left(  0\right)  /\Gamma\left(  0\right)  =\lim_{r\rightarrow0}%
\psi\left(  r\right)  /\Gamma\left(  r\right)  =-1$) and to the harmonic
Stieltjes constants $\gamma_{H}\left(  m\right)  $ when $r=1$. In Section
\ref{sec-ev}\textbf{ }we confer two more representations for $\gamma
_{h^{\left(  r\right)  }}\left(  m\right)  $ in addition to (\ref{GSC}). The
first one is in terms of Stieltjes constants $\gamma_{h^{\left(  r-1\right)
}}\left(  m\right)  ,$ $\gamma\left(  m\right)  $ and values related to the
zeta function (Theorem \ref{teorc}). The second one is in terms of Stieltjes
constants $\gamma_{H}\left(  m\right)  ,$ $\gamma\left(  m\right)  $ and some
other special values related to the Riemann zeta function (Theorem
\ref{teoghr}). In Section \ref{sec-*}, we consider the limit\textbf{ }
\[
\gamma_{h^{\left(  r\right)  }}^{\ast}\left(  m\right)  =\lim_{n\rightarrow
\infty}\left(  \sum_{k=1}^{n}\frac{h_{k}^{\left(  r\right)  }}{k^{r}}\ln
^{m}k-\int_{1}^{n}\frac{h_{x}^{\left(  r\right)  }}{x^{r}}\ln^{m}xdx\right)  .
\]
which is motivated by the interpretation of (\ref{SC})\textbf{ }as%
\begin{equation}
\lim_{n\rightarrow\infty}\left(  \sum_{k=1}^{n}f\left(  k\right)  -\int%
_{1}^{n}f\left(  x\right)  dx\right)  .\label{MaC}%
\end{equation}
(If $f:\left(  0,\infty\right)  \rightarrow\left(  0,\infty\right)  $ is
continuous, strictly decreasing and $\lim\limits_{x\rightarrow\infty}f\left(
x\right)  =0,$ then the limit (\ref{MaC}) exists (cf. \cite{We}).) Here
$h_{x}^{\left(  r\right)  }$ is an analytic extension of $h_{n}^{(r)}$,
defined by \cite{M}
\begin{equation}
h_{x}^{\left(  r\right)  }=\frac{x^{\overline{r}}}{x\Gamma\left(  r\right)
}\left(  \psi\left(  x+r\right)  -\psi\left(  r\right)  \right)  ,\text{
}r,x,r+x\in\mathbb{C}\backslash\left\{  0,-1,-2,...\right\}  ,\label{hxr}%
\end{equation}
where%
\begin{equation}
x^{\overline{r}}=x\left(  x+1\right)  \cdots\left(  x+r-1\right)
=\sum\limits_{j=0}^{r}%
\genfrac{[}{]}{0pt}{}{r}{j}%
x^{j}\label{6}%
\end{equation}
and $%
\genfrac{[}{]}{0pt}{}{r}{j}%
$ are the Stirling numbers of the first kind, with $%
\genfrac{[}{]}{0pt}{}{r}{r}%
=1$ for $r\geq0$\ and $%
\genfrac{[}{]}{0pt}{}{r}{0}%
=0$ for $r>0$.

Moreover, we give a representation for $\gamma_{h^{\left(  r\right)  }}^{\ast
}\left(  m\right)  $ (Theorem \ref{teomod}) which leads us to show that
integrals involving the generalized exponential function
\[
E_{s}^{0}\left(  t\right)  =\int_{1}^{\infty}\frac{e^{-xt}}{x^{s}}dx
\]
can be written in terms of some special constants (Theorem \ref{teo-exp}). The
following examples demonstrate Theorem \ref{teo-exp}:
\begin{align*}
\int_{0}^{\infty}\left(  \frac{1}{t}-\frac{1}{1-e^{-t}}\right)  E_{1}%
^{0}\left(  t\right)  dt  &  =-\gamma\left(  1\right)  -\sigma_{1}%
+\zeta\left(  2\right)  -1,\\
\int_{0}^{\infty}\left(  \frac{1}{t}-\frac{1}{1-e^{-t}}\right)  E_{2}%
^{0}\left(  t\right)  dt  &  =-\gamma+\sigma_{2}-\zeta^{\prime}\left(
2\right)  -\frac{3}{2},
\end{align*}
where%
\[
\sigma_{k}=\sum_{j=1}^{\infty}\frac{\left(  -1\right)  ^{j-1}}{j}\zeta\left(
k+j\right)  ,\text{ }k\geq1\text{ (see \cite{Bo,CC,CoC,DBA}).}%
\]

\section{Meromorphic continuation \label{sec-mteo}}

The following theorem, which was firstly proven by Kamano \cite[Theorem
1.1]{KK}, gives the meromorphic continuation of the function $\zeta
_{h^{\left(  r\right)  }}.$ Here we give a slightly different proof.

\begin{theorem}
\label{mteo}Let $r$ be a non-negative integer. The function $\zeta_{h^{\left(
r\right)  }}$ has a meromorphic continuation to the region\textbf{
}$\mathbb{C}\backslash\left\{  k\in\mathbb{Z}:k\leq r\right\}  $.
\end{theorem}

\begin{proof}
We first recall the following expansion of the harmonic numbers (see for
example \cite[Eq. (2)]{AV})\textbf{ }%
\begin{equation}
H_{n}=\ln n+\gamma+\frac{1}{2n}+\sum\limits_{m=1}^{k}\frac{\zeta\left(
1-2m\right)  }{n^{2m}}+\int_{n}^{\infty}\frac{P_{2k+1}\left(  x\right)
}{x^{2k+2}}dx\text{,} \label{3}%
\end{equation}
where $P_{k}\left(  x\right)  =B_{k}\left(  x-\left\lfloor x\right\rfloor
\right)  $ is the periodic extension of the Bernoulli polynomial $B_{k}\left(
x\right)  $ given by%
\[
\frac{te^{xt}}{e^{t}-1}=\sum\limits_{k=0}^{k}B_{k}\left(  x\right)
\frac{t^{k}}{k!},\text{ }|t|<2\pi
\]
with $\left\lfloor x\right\rfloor $ being the largest integer $\leq x$.

Considering the well-known relation $\psi\left(  n\right)  +\gamma=H_{n-1}%
$\ for $n\in\mathbb{N}$, (\ref{hxr}) becomes%
\begin{equation}
h_{n}^{\left(  r\right)  }=\frac{n^{\overline{r}}}{n\Gamma\left(  r\right)
}\left(  H_{n+r-1}-\psi\left(  r\right)  -\gamma\right)  . \label{2}%
\end{equation}
For $\operatorname{Re}\left(  s\right)  =\sigma>r,$ from (\ref{3}) and
(\ref{2}), we deduce the following representation for $\zeta_{h^{\left(
r\right)  }}\left(  s\right)  $:%
\begin{align}
\zeta_{h^{\left(  r\right)  }}\left(  s\right)   &  =-\frac{1}{\Gamma\left(
r\right)  }\zeta^{\prime}\left(  s+1-r\right)  -\frac{\psi\left(  r\right)
}{\Gamma\left(  r\right)  }\zeta\left(  s+1-r\right) \nonumber\\
&  +\frac{1}{\Gamma\left(  r\right)  }\sum\limits_{j=1}^{r-1}\left\{
\genfrac{[}{]}{0pt}{}{r}{j}%
\left(  -\zeta^{\prime}\left(  s+1-j\right)  -\psi\left(  r\right)
\zeta\left(  s+1-j\right)  \right)  \right. \nonumber\\
&  \qquad\left.  +%
\genfrac{[}{]}{0pt}{}{r}{j+1}%
\left(  \frac{1}{2}\zeta\left(  s+1-j\right)  +\sum\limits_{v=1}^{r-1}%
\sum_{n=1}^{\infty}\frac{1}{n^{s-j}\left(  n+v\right)  }\right)  \right\}
\nonumber\\
&  +\frac{1}{\Gamma\left(  r\right)  }\sum\limits_{j=1}^{r}%
\genfrac{[}{]}{0pt}{}{r}{j}%
\left(  \sum\limits_{m=1}^{k}\zeta\left(  1-2m\right)  \zeta\left(
s+1-j+2m\right)  +R\left(  s,k,j\right)  \right)  , \label{1}%
\end{align}
where
\[
R\left(  s,k,j\right)  =\sum_{n=1}^{\infty}\frac{1}{n^{s+1-j}}\int_{n}%
^{\infty}\frac{P_{2k+1}\left(  x\right)  }{x^{2k+2}}dx.
\]
For each $k\in\mathbb{N}$ the function\textbf{ }$P_{k}\left(  x\right)  $ is
bounded, say $\left\vert P_{k}\left(  x\right)  \right\vert \leq K$, then we
have
\[
\left\vert \sum_{n=1}^{\infty}\frac{1}{n^{s+1-j}}\int_{n}^{\infty}%
\frac{P_{2k+1}\left(  x\right)  }{x^{2k+2}}dx\right\vert \leq\frac{K}%
{2k+1}\sum_{n=1}^{\infty}\frac{1}{n^{\sigma-j+2k+2}}.
\]
Therefore $R\left(  s,k,j\right)  $ is analytic in the half-plane
$\sigma>r-2k-1$, and (\ref{1}) provides the analytic continuation of
$\zeta_{h^{\left(  r\right)  }}\left(  s\right)  $ in the half-plane
$\sigma>r-2k-1$.

It is seen from (\ref{1}) that the singularities of $\zeta_{h^{\left(
r\right)  }}\left(  s\right)  $ arise from the poles of $\zeta^{\prime}\left(
s+1-j\right)  ,$ $\zeta\left(  s+1-j\right)  $, $\sum_{n=1}^{\infty}\frac
{1}{n^{s-j}\left(  n+v\right)  }$ and $\zeta\left(  s+1-j+2m\right)  $.

For $r>0$ there is a second-order pole at $s=r$ arising from the term
$-\frac{1}{\Gamma\left(  r\right)  }\zeta^{\prime}\left(  s+1-r\right)
-\frac{\psi\left(  r\right)  }{\Gamma\left(  r\right)  }\zeta\left(
s+1-r\right)  $ with residue $-\frac{\psi\left(  r\right)  }{\Gamma\left(
r\right)  }$. Moreover, there are second-order poles at $s=j$ for $1\leq j\leq
r-1$, and simple poles at $s=j-2m$ with $j-2m<1$ for $1\leq j\leq r$ and
$1\leq m\leq k$ $.$ Since $k$ is an arbitrary positive integer, this implies
that $\zeta_{h^{\left(  r\right)  }}\left(  s\right)  $ has a simple pole at
every non-positive integer.
\end{proof}

\begin{remark}
Interested readers can find the residues of $\zeta_{h^{\left(  r\right)  }%
}\left(  s\right)  $ at $s=k,$ $k\in\mathbb{Z}$ with $k<r$\ in \cite[Theorem
1.1]{KK}. However, there are little misprints; the factors $r!$ in equations (1.6), (1.7) and (1.8) of \cite[Theorem 1.1]{KK} should be $\left(
r-1\right)  !$.
\end{remark}

\section{Laurent expansion and hyperharmonic Stieltjes
constants\label{sec-mteo2}}

This section is devoted to determine the hyperharmonic Stieltjes constants,
the main theme of this study.

\begin{theorem}
\label{mteo2}Let $r$ be a non-negative integer. The hyperharmonic zeta
function has the following Laurent expansion in the annulus $0<\left\vert
s-r\right\vert <1$%
\[
\zeta_{h^{\left(  r\right)  }}\left(  s\right)  =\frac{a_{-2}}{\left(
s-r\right)  ^{2}}+\frac{a_{-1}}{s-r}+\sum_{m=0}^{\infty}\frac{\left(
-1\right)  ^{m}\gamma_{h^{\left(  r\right)  }}\left(  m\right)  }{m!}\left(
s-r\right)  ^{m},
\]
where $a_{-2}=1/\Gamma\left(  r\right)  ,$ $a_{-1}=-\psi\left(  r\right)
/\Gamma\left(  r\right)  $ and
\[
\gamma_{h^{\left(  r\right)  }}\left(  m\right)  =\lim_{x\rightarrow\infty
}\left(  \sum_{n\leq x}\frac{h_{n}^{\left(  r\right)  }\log^{m}n}{n^{r}%
}-a_{-2}\frac{\ln^{m+2}x}{m+2}-a_{-1}\frac{\ln^{m+1}x}{m+1}\right)  .
\]

\end{theorem}

It is clear from the proof of Theorem \ref{mteo} that, $\zeta_{h^{\left(
r\right)  }}\left(  s\right)  $ has the Laurent series
\[
\zeta_{h^{\left(  r\right)  }}\left(  s\right)  =\frac{a_{-2}}{\left(
s-r\right)  ^{2}}+\frac{a_{-1}}{s-r}+\sum_{m=0}^{\infty}a_{m}\left(
s-r\right)  ^{m}%
\]
in the annulus $0<|s-r|<1,$ where $a_{-2}=1/\Gamma\left(  r\right)  $ and
$a_{-1}=-\psi\left(  r\right)  /\Gamma\left(  r\right)  $.

To complete the proof we must show that
\[
a_{m}=\frac{\left(  -1\right)  ^{m}}{m!}\lim_{x\rightarrow\infty}\left(
\sum_{n\leq x}\frac{h_{n}^{\left(  r\right)  }\log^{m}n}{n^{r}}-a_{-2}%
\frac{\ln^{m+2}x}{m+2}-a_{-1}\frac{\ln^{m+1}x}{m+1}\right)  .
\]
For this purpose we give some lemmas and a theorem.

\begin{lemma}
[Abel summation formula]\label{lem1}If $b_{1},b_{2},b_{3},\ldots$ is a
sequence of complex numbers and $v\left(  x\right)  $ has a continuous
derivative for $x>1$, then%
\[
\sum_{n\leq x}b_{n}v\left(  n\right)  =\left(  \sum_{n\leq x}b_{n}\right)
v\left(  x\right)  -\int\limits_{1}^{x}\left(  \sum_{n\leq t}b_{n}\right)
v^{\prime}\left(  t\right)  dt.
\]

\end{lemma}

For $n\leq x<n+1$ we see from (\ref{3}) and (\ref{2}) that
\begin{align*}
\sum_{k\leq x}h_{k}^{\left(  r\right)  }  &  =h_{n}^{\left(  r+1\right)
}=\frac{n^{\overline{r+1}}}{n\Gamma\left(  r+1\right)  }\left(  \log
n-\psi\left(  r+1\right)  +O\left(  \frac{1}{n}\right)  \right) \\
&  =\frac{1}{\Gamma\left(  r+1\right)  }n^{r}\log n-\frac{\psi\left(
r+1\right)  }{\Gamma\left(  r+1\right)  }n^{r}+O\left(  n^{r-1}\log n\right)
.
\end{align*}
Setting $b_{n}=h_{n}^{\left(  r\right)  }$ and $v\left(  x\right)  =x^{-s}$ in
Lemma \ref{lem1} gives the following result.

\begin{lemma}
\label{lem2}Let $r$ be a non-negative integer. For $\operatorname{Re}\left(
s\right)  >r$ we have%
\[
\zeta_{h^{\left(  r\right)  }}\left(  s\right)  =s\int\limits_{1}^{\infty
}x^{-s-1}\left(  \sum_{k\leq x}h_{k}^{\left(  r\right)  }\right)  dx.
\]

\end{lemma}

\begin{lemma}
\label{lem3}Let $r$ be a non-negative integer. Let%
\begin{equation}
E\left(  x\right)  =\sum_{k\leq x}h_{k}^{\left(  r\right)  }-\frac{1}%
{\Gamma\left(  r+1\right)  }x^{r}\log x+\frac{\psi\left(  r+1\right)  }%
{\Gamma\left(  r+1\right)  }x^{r}=O\left(  x^{r-1}\log x\right)  . \label{5}%
\end{equation}
Then for $\operatorname{Re}\left(  s\right)  >r-1$ we have%
\[
f\left(  s\right)  :=s\int\limits_{1}^{\infty}x^{-s-1}E\left(  x\right)
dx=\frac{\psi\left(  r+1\right)  }{\Gamma\left(  r+1\right)  }+\sum
_{n=0}^{\infty}a_{n}\left(  s-r\right)  ^{n}.
\]

\end{lemma}

\begin{proof}
For $\operatorname{Re}\left(  s\right)  >r-1$ the aforementioned integral is
an analytic function. Moreover for $\operatorname{Re}\left(  s\right)  >r$ we
have%
\begin{align*}
s\int\limits_{1}^{\infty}x^{-s-1}E\left(  x\right)  dx  &  =\zeta_{h^{\left(
r\right)  }}\left(  s\right)  -\frac{1}{\Gamma\left(  r+1\right)  }%
s\int\limits_{1}^{\infty}x^{r-s-1}\log xdx+\frac{\psi\left(  r+1\right)
}{\Gamma\left(  r+1\right)  }s\int\limits_{1}^{\infty}x^{r-s-1}dx\\
&  =\frac{a_{-2}}{\left(  s-r\right)  ^{2}}+\frac{a_{-1}}{s-r}+\sum
_{n=0}^{\infty}a_{n}\left(  s-r\right)  ^{n}-\frac{1}{\Gamma\left(
r+1\right)  }\frac{s}{\left(  r-s\right)  ^{2}}+\frac{\psi\left(  r+1\right)
}{\Gamma\left(  r+1\right)  }\frac{s}{s-r}\\
&  =\frac{1}{r^{2}\Gamma\left(  r\right)  }+\frac{\psi\left(  r\right)
}{r\Gamma\left(  r\right)  }+\sum_{n=0}^{\infty}a_{n}\left(  s-r\right)  ^{n}%
\end{align*}
from (\ref{5}) and Lemma \ref{lem2}. The proof is then completed.
\end{proof}

\begin{theorem}
\label{teo1}Let $m$ and $r$ be non-negative integers. Let $u<-\left(
r-1\right)  $. Then,%
\begin{align*}
\sum_{n\leq x}h_{n}^{\left(  r\right)  }n^{u}\log^{m}n  &  =\frac{1}%
{\Gamma\left(  r\right)  }\int\limits_{1}^{x}t^{r+u-1}\log^{m+1}tdt-\frac
{\psi\left(  r\right)  }{\Gamma\left(  r\right)  }\int\limits_{1}^{x}%
t^{r+u-1}\log^{m}tdt\\
&  -\frac{\psi\left(  r+1\right)  }{\Gamma\left(  r+1\right)  }\frac{d^{m}%
}{du^{m}}1+o\left(  1\right)  +\left(  -1\right)  ^{m}f^{\left(  m\right)
}\left(  -u\right)  ,
\end{align*}
where $f\left(  s\right)  $ is given in Lemma \ref{lem3}.
\end{theorem}

\begin{proof}
Set $v\left(  x\right)  =x^{u}\log^{m}x,$ $u<-r$ and $b_{n}=h_{n}^{\left(
r\right)  }$ in Lemma \ref{lem1}. Then%
\[
S:=\sum_{n\leq x}h_{n}^{\left(  r\right)  }n^{u}\log^{m}n=x^{u}\log^{m}%
x\sum_{n\leq x}h_{n}^{\left(  r\right)  }-\int\limits_{1}^{x}\sum_{n\leq
t}h_{n}^{\left(  r\right)  }\frac{d}{dt}\left(  t^{u}\log^{m}t\right)  dt.
\]
We now use the equality
\[
\frac{d}{dt}\left(  t^{u}\log^{m}t\right)  =\frac{d^{m}}{du^{m}}\left(
ut^{u-1}\right)
\]
to see that
\begin{align*}
S  &  =x^{u}\log^{m}x\sum_{n\leq x}h_{n}^{\left(  r\right)  }-\frac{d^{m}%
}{du^{m}}u\int\limits_{1}^{x}\sum_{n\leq t}h_{n}^{\left(  r\right)  }%
t^{u-1}dt\\
&  =x^{u}\log^{m}x\left(  \frac{1}{\Gamma\left(  r+1\right)  }x^{r}\log
x-\frac{\psi\left(  r+1\right)  }{\Gamma\left(  r+1\right)  }x^{r}+O\left(
x^{r-1}\log x\right)  \right) \\
&  -\frac{d^{m}}{du^{m}}u\int\limits_{1}^{x}\left(  E\left(  t\right)
+\frac{1}{\Gamma\left(  r+1\right)  }t^{r}\log t-\frac{\psi\left(  r+1\right)
}{\Gamma\left(  r+1\right)  }t^{r}\right)  t^{u-1}dt\\
&  =\frac{1}{\Gamma\left(  r+1\right)  }\left(  x^{u+r}\log^{m+1}x-\frac
{d^{m}}{du^{m}}u\int\limits_{1}^{x}t^{r+u-1}\log tdt\right) \\
&  -\frac{\psi\left(  r+1\right)  }{\Gamma\left(  r+1\right)  }\left(
x^{u+r}\log^{m}x-\frac{d^{m}}{du^{m}}u\int\limits_{1}^{x}t^{r+u-1}dt\right) \\
&  -\frac{d^{m}}{du^{m}}u\int\limits_{1}^{x}E\left(  t\right)  t^{u-1}%
dt+O\left(  x^{r-1+u}\log^{m+1}x\right)  .
\end{align*}
Here%
\begin{align*}
x^{u+r}\log^{m}x-\frac{d^{m}}{du^{m}}u\int\limits_{1}^{x}t^{r+u-1}dt  &
=\frac{d^{m}}{du^{m}}\left(  x^{u+r}-1-u\int\limits_{1}^{x}t^{r+u-1}dt\right)
+\frac{d^{m}}{du^{m}}1\\
&  =\frac{d^{m}}{du^{m}}\left(  r\int\limits_{1}^{x}t^{r+u-1}dt\right)
+\frac{d^{m}}{du^{m}}1\\
&  =r\int\limits_{1}^{x}t^{r+u-1}\log^{m}tdt+\frac{d^{m}}{du^{m}}1.
\end{align*}
Since%
\[
\int\limits_{1}^{x}t^{r+u-1}\log tdt=\frac{x^{u+r}\log x}{u+r}-\frac{x^{u+r}%
}{\left(  u+r\right)  ^{2}}+\frac{1}{\left(  u+r\right)  ^{2}},
\]
we obtain%
\[
x^{u+r}\log x-u\int\limits_{1}^{x}t^{r+u-1}\log tdt=r\int\limits_{1}%
^{x}t^{r+u-1}\log tdt+\frac{x^{u+r}}{u+r}-\frac{1}{u+r}.
\]
Hence,%
\begin{align*}
x^{u+r}\log^{m+1}x-\frac{d^{m}}{du^{m}}u\int\limits_{1}^{x}t^{r+u-1}\log tdt
&  =\frac{d^{m}}{du^{m}}\left(  x^{u+r}\log x-u\int\limits_{1}^{x}%
t^{r+u-1}\log tdt\right) \\
&  =\frac{d^{m}}{du^{m}}\left(  r\int\limits_{1}^{x}t^{r+u-1}\log
tdt+\int\limits_{1}^{x}t^{r+u-1}dt\right) \\
&  =r\int\limits_{1}^{x}t^{r+u-1}\log^{m+1}tdt+\int\limits_{1}^{x}%
t^{r+u-1}\log^{m}tdt.
\end{align*}
These complete the proof.
\end{proof}

Now we are ready to complete the proof of Theorem \ref{mteo2}.

\begin{proof}
[Proof of Theorem \ref{mteo2}]From Lemma \ref{lem3} we have
\[
f\left(  r\right)  =\frac{\psi\left(  r+1\right)  }{\Gamma\left(  r+1\right)
}+a_{0}\text{ and }a_{m}=\frac{f^{\left(  m\right)  }\left(  r\right)  }%
{m!}\text{ for }m>0.
\]
Setting $u=-r$ in Theorem \ref{teo1} yields
\[
a_{m}=\frac{f^{\left(  m\right)  }\left(  r\right)  }{m!}=\frac{\left(
-1\right)  ^{m}}{m!}\lim_{x\rightarrow\infty}\left(  \sum_{n\leq x}\frac
{h_{n}^{\left(  r\right)  }\log^{m}n}{n^{r}}-a_{-2}\frac{\ln^{m+2}x}%
{m+2}-a_{-1}\frac{\ln^{m+1}x}{m+1}\right)  ,
\]
which gives the desired result.
\end{proof}

\section{Alternative representations for hyperharmonic Stieltjes
constants\label{sec-ev}}

Recall that $\gamma_{h^{\left(  r\right)  }}\left(  m\right)  $\ reduces to
$\gamma\left(  m\right)  $ for $r=0$. Now we are going to analyze the case
$r>0$ in the following theorems.

\begin{theorem}
\label{teorc}Let $m$ be a non-negative integer and $r$ be a positive integer.
Then,%
\[
\gamma_{h^{\left(  r+1\right)  }}\left(  m\right)  =\frac{1}{r}\gamma
_{h^{\left(  r\right)  }}\left(  m\right)  -\frac{1}{r\Gamma\left(
r+1\right)  }\gamma\left(  m\right)   +\zeta_{h^{\left(  r\right)  }}^{\left(
m\right)  }\left(  r+1\right)  -\frac{1}{r\Gamma\left(  r+1\right)  }%
\sum\limits_{j=0}^{r-1}
\genfrac{[}{]}{0pt}{}{r}{j}
\zeta^{\left(  m\right)  }\left(  r+1-j\right)  ,
\]
where $\zeta_{h^{\left(  r\right)  }}^{\left(  m\right)  }\left(  r+1\right)
=\left.  \dfrac{d^{m}}{ds^{m}}\zeta_{h^{\left(  r\right)  }}\left(  s\right)
\right\vert _{s=r+1}$ and $\zeta^{\left(  m\right)  }\left(  r\right)
=\left.  \dfrac{d^{m}}{ds^{m}}\zeta\left(  s\right)  \right\vert _{s=r}$.
\end{theorem}

\begin{proof}
We employ the following equation \cite[Proposition 3]{DM}
\begin{align*}
h_{n}^{\left(  r+1\right)  }  &  =\left(  1+\frac{n}{r}\right)  h_{n}^{\left(
r\right)  }-\frac{n}{r\left(  n+r\right)  }\binom{n+r}{r}\\
&  =\left(  1+\frac{n}{r}\right)  h_{n}^{\left(  r\right)  }-\frac{1}%
{r\Gamma\left(  r+1\right)  }\sum\limits_{j=0}^{r}%
\genfrac{[}{]}{0pt}{}{r}{j}%
n^{j}%
\end{align*}
in (\ref{GSC}). After some manipulations, we deduce that%
\begin{align*}
\gamma_{h^{\left(  r+1\right)  }}\left(  m\right)   &  =\lim_{x\rightarrow
\infty}\left(  \frac{1}{r}\sum_{n\leq x}\frac{h_{n}^{\left(  r\right)  }%
\log^{m}n}{n^{r}}-\frac{1}{r\Gamma\left(  r\right)  }\frac{\ln^{m+2}x}%
{m+2}+\frac{\psi\left(  r\right)  }{r\Gamma\left(  r\right)  }\frac{\ln
^{m+1}x}{m+1}\right. \\
&  -\frac{1}{r\Gamma\left(  r+1\right)  }\sum_{n\leq x}\frac{\log^{m}n}%
{n}+\frac{1}{r\Gamma\left(  r+1\right)  }\frac{\ln^{m+1}x}{m+1}\\
&  \left.  +\sum_{n\leq x}\frac{h_{n}^{\left(  r\right)  }\log^{m}n}{n^{r+1}%
}-\frac{1}{r\Gamma\left(  r+1\right)  }\sum\limits_{j=0}^{r-1}%
\genfrac{[}{]}{0pt}{}{r}{j}%
\sum_{n\leq x}\frac{\log^{m}n}{n^{r+1-j}}\right)  ,
\end{align*}
which is the desired result.
\end{proof}

Now we are going to give a representation for the constants $\gamma
_{h^{\left(  r\right)  }}\left(  m\right)  $\ in terms of Stieltjes constants
$\gamma_{H}\left(  m\right)  ,$ $\gamma\left(  m\right)  $ and some other
special values related to the Riemann zeta function. For this purpose we need
the following lemma:

\begin{lemma}
\label{Top}Let $m$ and $r$ be non-negative integers. Then,%
\begin{align}
\sum_{k=1}^{n}\frac{h_{k}^{\left(  r\right)  }}{k^{r}}\ln^{m}k  &  =\frac
{1}{\Gamma\left(  r\right)  }\left(  \sum_{k=1}^{n}\frac{H_{k}}{k}\ln
^{m}k-\sum_{j=1}^{r-1}\frac{1}{j}\sum_{k=1}^{n+j}\frac{\ln^{m}k}{k}\right)
+\delta\left(  r\right)  \sum_{k=1}^{n}\frac{\ln^{m}k}{k}\nonumber\\
&  +\frac{1}{\Gamma\left(  r\right)  }\sum_{j=0}^{r-1}%
\genfrac{[}{]}{0pt}{}{r}{j}%
\left(  \sum_{k=1}^{n}\frac{H_{k+r-1}}{k^{r+1-j}}\ln^{m}k-H_{r-1}\sum
_{k=1}^{n}\frac{\ln^{m}k}{k^{r+1-j}}\right) \nonumber\\
&  +\frac{1}{\Gamma\left(  r\right)  }\sum_{j=1}^{r-1}\frac{1}{j}\left(
\sum_{k=1}^{j}\frac{\ln^{m}k}{k}-\sum_{k=j+1}^{n+j}\frac{\ln^{m}\left(
1-j/k\right)  }{k}\right)  , \label{top}%
\end{align}
where%
\[
\delta\left(  r\right)  =\frac{1}{\Gamma\left(  r\right)  }\sum_{j=1}%
^{r-1}\frac{1}{j}-\frac{\psi\left(  r\right)  +\gamma}{\Gamma\left(  r\right)
}=\left\{
\begin{array}
[c]{cc}%
1, & r=0,\\
0, & r>0.
\end{array}
\right.
\]

\end{lemma}

\begin{proof}
In the light of (\ref{6}) and (\ref{2}) we have%
\[
h_{k}^{\left(  r\right)  }=\frac{1}{\Gamma\left(  r\right)  }\sum_{j=0}^{r}%
\genfrac{[}{]}{0pt}{}{r}{j}%
k^{j-1}\left(  H_{k+r-1}-\psi\left(  r\right)  -\gamma\right)  ,
\]
from which we obtain that%
\begin{align*}
\Gamma\left(  r\right)  \sum_{k=1}^{n}\frac{h_{k}^{\left(  r\right)  }}%
{k^{s}}  &  =\sum_{j=0}^{r-1}%
\genfrac{[}{]}{0pt}{}{r}{j}%
\sum_{k=1}^{n}\frac{H_{k+r-1}}{k^{s+1-j}}+\sum_{k=1}^{n}\frac{H_{k+r-1}%
}{k^{s+1-r}}\\
&  -\left(  \psi\left(  r\right)  +\gamma\right)  \left(  \sum_{k=1}^{n}%
\frac{1}{k^{s+1-r}}+\sum_{j=0}^{r-1}%
\genfrac{[}{]}{0pt}{}{r}{j}%
\sum_{k=1}^{n}\frac{1}{k^{s+1-j}}\right)  .
\end{align*}
Differentiating both sides $m$ times with respect to $s$ at $s=r$ gives%
\begin{align}
\Gamma\left(  r\right)  \sum_{k=1}^{n}\frac{h_{k}^{\left(  r\right)  }}{k^{s}%
}\ln^{m}k  &  =\sum_{j=0}^{r-1}%
\genfrac{[}{]}{0pt}{}{r}{j}%
\sum_{k=1}^{n}\frac{H_{k+r-1}}{k^{r+1-j}}\ln^{m}k+\sum_{k=1}^{n}%
\frac{H_{k+r-1}}{k}\ln^{m}k\nonumber\\
&  -\left(  \psi\left(  r\right)  +\gamma\right)  \left(  \sum_{k=1}^{n}%
\frac{\ln^{m}k}{k}+\sum_{j=0}^{r-1}%
\genfrac{[}{]}{0pt}{}{r}{j}%
\sum_{k=1}^{n}\frac{\ln^{m}k}{k^{r+1-j}}\right)  . \label{4}%
\end{align}
It is easy to see that
\begin{equation}
\sum_{k=1}^{n}\frac{H_{k+r-1}}{k}\ln^{m}k=\sum_{k=1}^{n}\frac{H_{k}}{k}\ln
^{m}k+\sum_{j=1}^{r-1}\frac{1}{j}\sum_{k=1}^{n}\frac{\ln^{m}k}{k}-\sum
_{k=1}^{n}\sum_{j=1}^{r-1}\frac{\ln^{m}k}{j\left(  k+j\right)  }\nonumber
\end{equation}
and
\begin{align*}
\sum_{k=1}^{n}\sum_{j=1}^{r-1}\frac{\ln^{m}k}{j\left(  k+j\right)  }  &
=\sum_{j=1}^{r-1}\sum_{k=j+1}^{n+j}\frac{\ln^{m}k+\ln^{m}\left(  1-j/k\right)
}{jk}\\
&  =\sum_{j=1}^{r-1}\frac{1}{j}\left(  \sum_{k=1}^{n+j}\frac{\ln^{m}k}{k}%
-\sum_{k=1}^{j}\frac{\ln^{m}k}{k}+\sum_{k=j+1}^{n+j}\frac{\ln^{m}\left(
1-j/k\right)  }{k}\right)  .
\end{align*}
Then, we have%
\begin{align}
\sum_{k=1}^{n}\frac{H_{k+r-1}}{k}\ln^{m}k  &  =\sum_{k=1}^{n}\frac{H_{k}}%
{k}\ln^{m}k+H_{r-1}\sum_{k=1}^{n}\frac{\ln^{m}k}{k}-\sum_{j=1}^{r-1}\frac
{1}{j}\sum_{k=1}^{n+j}\frac{\ln^{m}k}{k}\nonumber\\
&  +\sum_{j=1}^{r-1}\frac{1}{j}\sum_{k=1}^{j}\frac{\ln^{m}k}{k}-\sum
_{j=1}^{r-1}\sum_{k=j+1}^{n+j}\frac{\ln^{m}\left(  1-j/k\right)  }{jk}
\label{top1}%
\end{align}
Hence, (\ref{top}) follows from (\ref{4}) and (\ref{top1}).
\end{proof}

We are ready to give the aforementioned representation for the constants
$\gamma_{h^{\left(  r\right)  }}\left(  m\right)  .$

\begin{theorem}
\label{teoghr}Let $m$ and $r$ be positive integers. Then,%
\begin{align*}
\Gamma\left(  r\right)  \gamma_{h^{\left(  r\right)  }}\left(  m\right)   &
=\gamma_{H}\left(  m\right)  -H_{r-1}\gamma\left(  m\right) \\
&  +\sum_{j=0}^{r-1}
\genfrac{[}{]}{0pt}{}{r}{j}
\left(  \widetilde{\zeta}_{H}^{\left(  m\right)  }\left(  r+1-j,r-1\right)
-H_{r-1}\zeta^{\left(  m\right)  }\left(  r+1-j\right)  \right)  +\sum
_{j=1}^{r-1}\frac{1}{j}\left(  \sum_{k=1}^{j}\frac{\ln^{m}k}{k}-C\left(
j,m\right)  \right)  ,
\end{align*}
where%
\[
\widetilde{\zeta}_{H}\left(  s,a\right)  =\sum_{k=1}^{\infty}\frac{H_{k+a}%
}{k^{s}}=\sum_{k=a+1}^{\infty}\frac{H_{k}}{\left(  k-a\right)  ^{s}},\text{
}\widetilde{\zeta}_{H}^{\left(  m\right)  }\left(  r,a\right)  =\left.
\dfrac{d^{m}}{ds^{m}}\widetilde{\zeta}_{H}\left(  s,a\right)  \right\vert
_{s=r}%
\]
and the constants $C\left(  j,m\right)  $
\[
C\left(  j,m\right)  =\sum_{k=j+1}^{\infty}\frac{\ln^{m}\left(  1-j/k\right)
}{k}.
\]

\end{theorem}

\begin{proof}
From Lemma \ref{Top} we have%
\begin{align*}
&  \sum_{k=1}^{n}\frac{h_{k}^{\left(  r\right)  }}{k^{r}}\ln^{m}k-\frac
{1}{\Gamma\left(  r\right)  }\frac{\ln^{m+2}n}{m+2}+\frac{\psi\left(
r\right)  }{\Gamma\left(  r\right)  }\frac{\ln^{m+1}n}{m+1}\\
&  =\frac{1}{\Gamma\left(  r\right)  }\left(  \sum_{k=1}^{n}\frac{H_{k}}{k}%
\ln^{m}k-\sum_{j=1}^{r-1}\frac{1}{j}\sum_{k=1}^{n+j}\frac{\ln^{m}k}{k}%
-\frac{\ln^{m+2}n}{m+2}+\psi\left(  r\right)  \frac{\ln^{m+1}n}{m+1}\right) \\
&  +\frac{1}{\Gamma\left(  r\right)  }\sum_{j=0}^{r-1}%
\genfrac{[}{]}{0pt}{}{r}{j}%
\left(  \sum_{k=1}^{n}\frac{H_{k+r-1}}{k^{r+1-j}}\ln^{m}k-H_{r-1}\sum
_{k=1}^{n}\frac{\ln^{m}k}{k^{r+1-j}}\right) \\
&  +\frac{1}{\Gamma\left(  r\right)  }\sum_{j=1}^{r-1}\frac{1}{j}\left(
\sum_{k=1}^{j}\frac{\ln^{m}k}{k}-\sum_{k=j+1}^{n+j}\frac{\ln^{m}\left(
1-j/k\right)  }{k}\right)  .
\end{align*}
By letting $n$ tend to infinity we see that%
\begin{align*}
\gamma_{h^{\left(  r\right)  }}\left(  m\right)   &  =\frac{1}{\Gamma\left(
r\right)  }\lim_{n\rightarrow\infty}\left(  \sum_{k=1}^{n}\frac{H_{k}}{k}%
\ln^{m}k-\sum_{j=1}^{r-1}\frac{1}{j}\sum_{k=1}^{n+j}\frac{\ln^{m}k}{k}%
-\frac{\ln^{m+2}n}{m+2}+\psi\left(  r\right)  \frac{\ln^{m+1}n}{m+1}\right) \\
&  +\frac{1}{\Gamma\left(  r\right)  }\sum_{j=0}^{r-1}%
\genfrac{[}{]}{0pt}{}{r}{j}%
\left(  \widetilde{\zeta}_{H}^{\left(  m\right)  }\left(  r+1-j,r-1\right)
-H_{r-1}\zeta^{\left(  m\right)  }\left(  r+1-j\right)  \right) \\
&  +\frac{1}{\Gamma\left(  r\right)  }\sum_{j=1}^{r-1}\frac{1}{j}\left(
\sum_{k=1}^{j}\frac{\ln^{m}k}{k}-\sum_{k=j+1}^{\infty}\frac{\ln^{m}\left(
1-j/k\right)  }{k}\right)  .
\end{align*}
Then the desired result follows from%
\begin{align*}
&  \lim_{n\rightarrow\infty}\left(  \sum_{k=1}^{n}\frac{H_{k}}{k}\ln^{m}%
k-\sum_{j=1}^{r-1}\frac{1}{j}\sum_{k=1}^{n+j}\frac{\ln^{m}k}{k}-\frac
{\ln^{m+2}n}{m+2}+\psi\left(  r\right)  \frac{\ln^{m+1}n}{m+1}\right) \\
&  =\lim_{n\rightarrow\infty}\left(  \gamma_{h^{\left(  1\right)  }}\left(
m\right)  +H_{r-1}\frac{\ln^{m+1}n}{m+1}-\sum_{j=1}^{r-1}\frac{1}{j}\sum
_{k=1}^{n+j}\frac{\ln^{m}k}{k}\right) \\
&  =\gamma_{h^{\left(  1\right)  }}\left(  m\right)  -\sum_{j=1}^{r-1}\frac
{1}{j}\lim_{n\rightarrow\infty}\left(  \frac{\ln^{m+1}\left(  1+j/n\right)
}{m+1}+\sum_{k=1}^{n+j}\frac{\ln^{m}k}{k}-\frac{\ln^{m+1}\left(  n+j\right)
}{m+1}\right) \\
&  =\gamma_{h^{\left(  1\right)  }}\left(  m\right)  -H_{r-1}\gamma
_{h^{\left(  0\right)  }}\left(  m\right)  .
\end{align*}

\end{proof}

\section{A formal extension of the Stieltjes constants\label{sec-*}}

In this section, we give a presentation for the constants
\[
\gamma_{h^{\left(  r\right)  }}^{\ast}\left(  m\right)  =\lim_{n\rightarrow
\infty}\left(  \sum_{k=1}^{n}\frac{h_{k}^{\left(  r\right)  }}{k^{r}}\ln
^{m}k-\int_{1}^{n}\frac{h_{x}^{\left(  r\right)  }}{x^{r}}\ln^{m}xdx\right)
,
\]
which are a formal extension of the Stieltjes constants. As a result we show
that integrals involving generalized exponential function can be written in
terms of some special constants.

We first analyze the integral in the limit above.

\begin{lemma}
\label{Int}Let $m$ and $r$ be non-negative integers. Then,%
\begin{align}
\Gamma\left(  r\right)  \int_{1}^{n}\frac{h_{x}^{\left(  r\right)  }}{x^{r}%
}\ln^{m}xdx  & =\frac{\ln^{m+2}n}{m+2}-\psi\left(  r\right)  \frac{\ln^{m+1}%
n}{m+1}+\frac{n^{\overline{r}}}{n^{r+1}}\ln^{m}n\nonumber\\
&  +\int_{1}^{n}\left(
{\displaystyle\sum_{j=0}^{r-1}}
\genfrac{[}{]}{0pt}{0}{r}{j}%
\frac{\ln^{m+1}x-\psi\left(  r\right)  \ln^{m}x}{x^{r+1-j}}+%
{\displaystyle\sum_{j=0}^{r}}
\genfrac{[}{]}{0pt}{0}{r}{j}%
\frac{\left(  r+1\right)  \ln^{m}x-m\ln^{m-1}x}{x^{r+2-j}}\right)
dx\nonumber\\
&  +%
{\displaystyle\sum_{j=0}^{r}}
\genfrac{[}{]}{0pt}{0}{r}{j}%
\int_{0}^{\infty}\left(  \frac{1}{t}-\frac{1}{1-e^{-t}}\right)  \left(
\int_{1}^{n}\frac{\ln^{m}x}{x^{r+1-j}}e^{-xt}dx\right)  dt. \label{int}%
\end{align}

\end{lemma}

\begin{proof}
Utilizing the identity $\psi\left(  x+1\right)  =\psi\left(  x\right)  +1/x$
\ in (\ref{hxr}) we have
\begin{align}
\Gamma\left(  r\right)  \int_{1}^{n}\frac{h_{x}^{\left(  r\right)  }}{x^{s}%
}dx  &  =\int_{1}^{n}\frac{x^{\overline{r}}}{x^{s+1}}\psi\left(  x+1\right)
dx\nonumber\\
&  +\int_{1}^{n}\frac{1}{x^{s}}\frac{d}{dx}\left(  x+1\right)  ^{\overline
{r-1}}dx-\psi\left(  r\right)  \int_{1}^{n}\frac{x^{\overline{r}}}{x^{s+1}}dx.
\label{inta}%
\end{align}
In view of (\ref{6})\ the first integral on the RHS of (\ref{inta})\ becomes\
\[
\int_{1}^{n}\frac{x^{\overline{r}}}{x^{s+1}}\psi\left(  x+1\right)  dx=%
{\displaystyle\sum_{j=0}^{r}}
\genfrac{[}{]}{0pt}{0}{r}{j}%
\int_{1}^{n}\frac{\psi\left(  x+1\right)  }{x^{s+1-j}}dx.
\]
Thanks to the expression \cite[p. 26]{SC}
\[
\psi\left(  z\right)  =\ln z+\int_{0}^{\infty}\left(  \frac{1}{t}-\frac
{1}{1-e^{-t}}\right)  e^{-zt}dt,\text{ }\operatorname{Re}\left(  z\right)
>0,
\]
we write
\begin{align*}
\int_{1}^{n}\frac{\psi\left(  x+1\right)  }{x^{s+1-j}}dx  &  =\int_{1}%
^{n}\frac{\ln x}{x^{s+1-j}}dx+\int_{1}^{n}\frac{1}{x^{s+2-j}}dx\\
&  +\int_{0}^{\infty}\left(  \frac{1}{t}-\frac{1}{1-e^{-t}}\right)  \left(
\int_{1}^{n}\frac{e^{-xt}}{x^{s+1-j}}dx\right)  dt.
\end{align*}
We then deduce that
\begin{align*}
\left.  \frac{\partial^{m}}{\partial s^{m}}\int_{1}^{n}\frac{x^{\overline{r}}%
}{x^{s+1}}\psi\left(  x+1\right)  dx\right\vert _{s=r}  & = {\displaystyle\sum
_{j=0}^{r}}
\genfrac{[}{]}{0pt}{0}{r}{j}
\int_{1}^{n}\frac{\psi\left(  x+1\right)  }{x^{r+1-j}}\ln^{m}xdx\\
&  =\int_{1}^{n}\frac{\ln^{m+1}x}{x}dx+ {\displaystyle\sum_{j=0}^{r-1}}
\genfrac{[}{]}{0pt}{0}{r}{j}%
\int_{1}^{n}\frac{\ln^{m+1}x}{x^{r+1-j}}dx+ {\displaystyle\sum_{j=0}^{r}}
\genfrac{[}{]}{0pt}{0}{r}{j}
\int_{1}^{n}\frac{\ln^{m}x}{x^{r+2-j}}dx\\
&  +{\displaystyle\sum_{j=0}^{r}}
\genfrac{[}{]}{0pt}{0}{r}{j}
\int_{0}^{\infty}\left(  \frac{1}{t}-\frac{1}{1-e^{-t}}\right)  \left(
\int_{1}^{n}\frac{\ln^{m}x}{x^{r+1-j}}e^{-xt}dx\right)  dt.
\end{align*}
On the other hand, the second and the third integrals on the RHS of
(\ref{inta}) can be obtained as
\begin{align*}
\left.  \frac{\partial^{m}}{\partial s^{m}}\int_{1}^{n}\frac{1}{x^{s}}\frac
{d}{dx}\left(  x+1\right)  ^{\overline{r-1}}dx\right\vert _{s=r}  &  =\int%
_{1}^{n}\frac{\ln^{m}x}{x^{r}}\frac{d}{dx}\left(  x+1\right)  ^{\overline
{r-1}}dx\\
&  =\left(  n+1\right)  ^{\overline{r-1}}\frac{\ln^{m}n}{n^{r}}-%
{\displaystyle\sum_{j=0}^{r}}
\genfrac{[}{]}{0pt}{0}{r}{j}%
\int_{1}^{n}\frac{\left(  -r\ln^{m}x+m\ln^{m-1}x\right)  }{x^{r+2-j}}dx
\end{align*}
and%
\[
\left.  \psi\left(  r\right)  \frac{\partial^{m}}{\partial s^{m}}\int_{1}%
^{n}\frac{x^{\overline{r}}}{x^{s+1}}dx\right\vert _{s=r}=\psi\left(  r\right)
\int_{1}^{n}\frac{\ln^{m}x}{x}dx+\psi\left(  r\right)
{\displaystyle\sum_{j=0}^{r-1}}
\genfrac{[}{]}{0pt}{0}{r}{j}%
\int_{1}^{n}\frac{\ln^{m}x}{x^{r+1-j}}dx.
\]
Combining the results above yields (\ref{int}).
\end{proof}

The next theorem gives the aforementioned representation for the constants
$\gamma_{h^{\left(  r\right)  }}^{\ast}\left(  m\right)  $, which involves the
generalized integro-exponential function (e.g. \cite{Mil})
\[
E_{s}^{m}\left(  t\right)  =\frac{1}{\Gamma\left(  m+1\right)  }\int%
_{1}^{\infty}\frac{e^{-xt}}{x^{s}}\ln^{m}xdx
\]
(see \cite{Mi,Hu} for different fields where the integral arises).

\begin{theorem}
\label{teomod}Let $m$ and $r$ be non-negative integers. Then,%
\begin{align*}
\gamma_{h^{\left(  r\right)  }}^{\ast}\left(  m\right)   &  =\gamma
_{h^{\left(  r\right)  }}\left(  m\right)  -\frac{m!}{\Gamma\left(  r\right)
}%
{\displaystyle\sum_{j=0}^{r-1}}
\genfrac{[}{]}{0pt}{0}{r}{j}%
\left(  \frac{m+1-\left(  r-j\right)  \psi\left(  r\right)  }{\left(
r-j\right)  ^{m+2}}+\frac{j}{\left(  r+1-j\right)  ^{m+1}}\right) \\
&  -r\frac{m!}{\Gamma\left(  r\right)  }-\frac{m!}{\Gamma\left(  r\right)  }%
{\displaystyle\sum_{j=0}^{r}}
\genfrac{[}{]}{0pt}{0}{r}{j}%
\int_{0}^{\infty}\left(  \frac{1}{t}-\frac{1}{1-e^{-t}}\right)  E_{r+1-j}%
^{m}\left(  t\right)  dt.
\end{align*}

\end{theorem}

\begin{proof}
In the light of Lemma \ref{Int}, we have%
\begin{align}
\gamma_{h^{\left(  r\right)  }}^{\ast}\left(  m\right)   &  =\lim
_{n\rightarrow\infty}\left(  \sum_{k=1}^{n}\frac{h_{k}^{\left(  r\right)  }%
}{k^{r}}\ln^{m}k-\frac{1}{\Gamma\left(  r\right)  }\frac{\ln^{m+2}n}%
{m+2}+\frac{\psi\left(  r\right)  }{\Gamma\left(  r\right)  }\frac{\ln^{m+1}%
n}{m+1}\right) \nonumber\\
&  -\frac{1}{\Gamma\left(  r\right)  }%
{\displaystyle\sum_{j=0}^{r-1}}
\genfrac{[}{]}{0pt}{0}{r}{j}%
\int_{1}^{\infty}\frac{\ln^{m+1}x-\psi\left(  r\right)  \ln^{m}x}{x^{r+1-j}%
}dx\nonumber\\
&  -\frac{1}{\Gamma\left(  r\right)  }%
{\displaystyle\sum_{j=0}^{r}}
\genfrac{[}{]}{0pt}{0}{r}{j}%
\int_{1}^{\infty}\frac{\left(  r+1\right)  \ln^{m}x-m\ln^{m-1}x}{x^{r+2-j}%
}dx\nonumber\\
&  -\frac{1}{\Gamma\left(  r\right)  }%
{\displaystyle\sum_{j=0}^{r}}
\genfrac{[}{]}{0pt}{0}{r}{j}%
\int_{0}^{\infty}\left(  \frac{1}{t}-\frac{1}{1-e^{-t}}\right)  \left(
\int_{1}^{\infty}\frac{\ln^{m}x}{x^{r+1-j}}e^{-xt}dx\right)  dt. \label{12}%
\end{align}

Using the following reduction formula
\[
I\left(  \mu,k\right)  =\int_{1}^{\infty}\frac{\ln^{\mu}x}{x^{k}}dx=\frac{\mu
}{k-1}I\left(  \mu-1,k\right)
\]
we deduce that%
\[
I\left(  \mu,k\right)  =\frac{\mu!}{\left(  k-1\right)  ^{\mu}}I\left(
0,k\right)  =\frac{\mu!}{\left(  k-1\right)  ^{\mu+1}}.
\]
Hence after some rearrangements, (\ref{12}) becomes
\begin{align*}
\gamma_{h^{\left(  r\right)  }}^{\ast}\left(  m\right)   &  =\gamma
_{h^{\left(  r\right)  }}\left(  m\right)  -\frac{m!}{\Gamma\left(  r\right)
}%
{\displaystyle\sum_{j=0}^{r-1}}
\genfrac{[}{]}{0pt}{0}{r}{j}%
\left(  \frac{m+1}{\left(  r-j\right)  ^{m+2}}-\frac{\psi\left(  r\right)
}{\left(  r-j\right)  ^{m+1}}\right) \\
&  -\frac{m!}{\Gamma\left(  r\right)  }%
{\displaystyle\sum_{j=0}^{r}}
\genfrac{[}{]}{0pt}{0}{r}{j}%
\left(  \frac{r+1}{\left(  r+1-j\right)  ^{m+1}}-\frac{1}{\left(
r+1-j\right)  ^{m}}\right) \\
&  -\frac{m!}{\Gamma\left(  r\right)  }%
{\displaystyle\sum_{j=0}^{r}}
\genfrac{[}{]}{0pt}{0}{r}{j}%
\int_{0}^{\infty}\left(  \frac{1}{t}-\frac{1}{1-e^{-t}}\right)  E_{r+1-j}%
^{m}\left(  t\right)  dt,
\end{align*}
which completes the proof.
\end{proof}

We conclude the paper with result on evaluation of the integrals involving
generalized exponential function.

\begin{theorem}
\label{teo-exp}Let $p\in\mathbb{N}$. Then the integral%
\[
\int_{0}^{\infty}\left(  \frac{1}{t}-\frac{1}{1-e^{-t}}\right)  E_{p}%
^{0}\left(  t\right)  dt
\]
can be written as a finite combination of the Riemann zeta values
$\zeta\left(  l\right)  $ and $\zeta^{\prime}\left(  k\right)  $,
Euler-Mascheroni constant $\gamma,$ Stieltjes constant $\gamma\left(
1\right)  $\ and the constants $\sigma_{k}.$
\end{theorem}

\begin{proof}
In special case $m=0,$ the constants $\gamma_{h^{\left(  r\right)  }}^{\ast
}\left(  0\right)  $ reduce to the constants
\[
\gamma_{h^{\left(  r\right)  }}^{\ast}=\lim_{n\rightarrow\infty}\left(
\sum_{k=1}^{n}\frac{h_{k}^{\left(  r\right)  }}{k^{r}}-\int_{1}^{n}\frac
{h_{x}^{\left(  r\right)  }}{x^{r}}dx\right)
\]
introduced in recent paper \cite{CDKCG}. Accordingly, we have two alternative
representations for $\gamma_{h^{\left(  r\right)  }}^{\ast}\left(  0\right)
=\gamma_{h^{\left(  r\right)  }}^{\ast}$. The first one appears from Theorem
\ref{teomod}:
\begin{align*}
\Gamma\left(  r\right)  \gamma_{h^{\left(  r\right)  }}^{\ast}  &
=\Gamma\left(  r\right)  \gamma_{h^{\left(  r\right)  }}\left(  0\right)  -%
{\displaystyle\sum_{j=0}^{r-1}}
\genfrac{[}{]}{0pt}{0}{r}{j}%
\left(  \frac{1-\left(  r-j\right)  \psi\left(  r\right)  }{\left(
r-j\right)  ^{2}}+\frac{j}{\left(  r+1-j\right)  }\right) \\
&  -r-%
{\displaystyle\sum_{j=0}^{r}}
\genfrac{[}{]}{0pt}{0}{r}{j}%
\int_{0}^{\infty}\left(  \frac{1}{t}-\frac{1}{1-e^{-t}}\right)  E_{r+1-j}%
^{0}\left(  t\right)  dt,
\end{align*}
and the second one is given in Theorem 8 of \cite{CDKCG}:
\begin{align*}
&  \Gamma\left(  r\right)  \gamma_{h^{\left(  r\right)  }}^{\ast}=\frac{1}%
{2}\gamma^{2}-\frac{1}{2}\zeta\left(  2\right)  +\sigma_{1}+\gamma\left(
1\right)  +\sum_{j=1}^{r-1}\frac{H_{j}}{j}-\left(  \psi\left(  r\right)
+\gamma\right)  \gamma+r!\\
&  -r+\sum_{j=1}^{r-1}%
\genfrac{[}{]}{0pt}{}{r}{j}%
\left\{
\begin{array}
[c]{l}%
\frac{r+3-j}{2}\zeta\left(  r+2-j\right)  -\frac{1}{2}\sum\limits_{v=1}%
^{r-j-1}\zeta\left(  r+1-j-v\right)  \zeta\left(  v+1\right) \\
-\sum\limits_{v=2}^{r-j}\left(  -1\right)  ^{v}\zeta\left(  r+2-j-v\right)
\left(  H_{r-1}^{\left(  v\right)  }+\dfrac{\left(  -1\right)  ^{r-j}}%
{v-1}\right)  +\frac{H_{r-1}}{r-j}\\
+\left(  -1\right)  ^{r-j}\left(  \sigma_{r+1-j}-\zeta^{\prime}\left(
r+1-j\right)  +\sum\limits_{v=1}^{r-1}\frac{H_{v}}{v^{r+1-j}}\right)
-\frac{r}{r+1-j}%
\end{array}
\right\}
\end{align*}
where $H_{p}^{\left(  v\right)  }=\sum_{k=1}^{p}k^{-v}$ are the generalized
harmonic numbers. As a consequence of these two representations we have%
\begin{align}
&
{\displaystyle\sum_{j=0}^{r}}
\genfrac{[}{]}{0pt}{0}{r}{j}%
\int_{0}^{\infty}\left(  \frac{1}{t}-\frac{1}{1-e^{-t}}\right)  E_{r+1-j}%
^{0}\left(  t\right)  dt\nonumber\\
&  =\Gamma\left(  r\right)  \gamma_{h^{\left(  r\right)  }}\left(  0\right)
-\frac{1}{2}\gamma^{2}+\frac{1}{2}\zeta\left(  2\right)  -\sigma_{1}%
-\gamma\left(  1\right)  -\sum_{j=1}^{r-1}\frac{H_{j}}{j}+\left(  \psi\left(
r\right)  +\gamma\right)  \gamma-r!\nonumber\\
&  -\sum_{j=1}^{r-1}%
\genfrac{[}{]}{0pt}{}{r}{j}%
\left\{
\begin{array}
[c]{l}%
\frac{r+3-j}{2}\zeta\left(  r+2-j\right)  -\frac{1}{2}\sum\limits_{v=1}%
^{r-j-1}\zeta\left(  r+1-j-v\right)  \zeta\left(  v+1\right) \\
-\sum\limits_{v=2}^{r-j}\left(  -1\right)  ^{v}\zeta\left(  r+2-j-v\right)
\left(  H_{r-1}^{\left(  v\right)  }+\dfrac{\left(  -1\right)  ^{r-j}}%
{v-1}\right)  +\frac{H_{r-1}}{r-j}\\
+\left(  -1\right)  ^{r-j}\left(  \sigma_{r+1-j}-\zeta^{\prime}\left(
r+1-j\right)  +\sum\limits_{v=1}^{r-1}\frac{H_{v}}{v^{r+1-j}}\right)
-\frac{r}{r+1-j}%
\end{array}
\right\} \nonumber\\
&  -%
{\displaystyle\sum_{j=0}^{r-1}}
\genfrac{[}{]}{0pt}{0}{r}{j}%
\left(  \frac{1-\left(  r-j\right)  \psi\left(  r\right)  }{\left(
r-j\right)  ^{2}}+\frac{j}{r+1-j}\right)  . \label{int-ie}%
\end{align}
To complete the proof it is enough to show that the constants
\[
\gamma_{h^{\left(  r\right)  }}:=\gamma_{h^{\left(  r\right)  }}\left(
0\right)  =\lim_{x\rightarrow\infty}\left(  \sum_{n\leq x}\frac{h_{n}^{\left(
r\right)  }}{n^{r}}-\frac{1}{2\Gamma\left(  r\right)  }\ln^{2}x+\frac
{\psi\left(  r\right)  }{\Gamma\left(  r\right)  }\ln x\right)
\]
can be written in terms of zeta values, Euler-Mascheroni constant and harmonic
numbers. To achieve this we appeal the following expression \cite[Lemma
4]{CDKCG}:
\begin{align*}
\sum_{k=1}^{n}\frac{h_{k}^{\left(  r\right)  }}{k^{r}}  &  =\frac{\left(
H_{n}\right)  ^{2}+H_{n}^{\left(  2\right)  }}{2\Gamma\left(  r\right)
}-\frac{\psi\left(  r\right)  +\gamma}{\Gamma\left(  r\right)  }H_{n}\\
&  +\frac{1}{\Gamma\left(  r\right)  }\sum_{j=0}^{r-1}%
\genfrac{[}{]}{0pt}{}{r}{j}%
\left(  \sum_{k=1}^{n}\frac{H_{k+r-1}}{k^{r+1-j}}-H_{r-1}H_{n}^{\left(
r+1-j\right)  }\right) \\
&  +\frac{1}{\Gamma\left(  r\right)  }\sum_{j=1}^{r-1}\frac{H_{j}}{j}%
-\frac{H_{r-1}}{\Gamma\left(  r\right)  }\left(  H_{n+r-1}-H_{n}\right)
+\frac{1}{\Gamma\left(  r\right)  }\sum_{j=1}^{r-1}\frac{H_{j-1}}{j+n},
\end{align*}
where $r$ is a non-negative integer.

We now utilize the asymptotic expression (cf. (\ref{3}))%
\[
H_{n}=\ln n+\gamma+O\left(  n^{-1}\right)
\]
to find that%
\[
\lim_{x\rightarrow\infty}\left(  \frac{\left(  H_{n}\right)  ^{2}}{2}-\left(
\psi\left(  r\right)  +\gamma\right)  H_{n}-\frac{\ln^{2}n}{2}+\psi\left(
r\right)  \ln n\right)  =-\frac{\gamma^{2}}{2}-\gamma\psi\left(  r\right)  .
\]
Hence,
\begin{align}
\Gamma\left(  r\right)  \gamma_{h^{\left(  r\right)  }}  &  =-\frac{\gamma
^{2}}{2}-\gamma\psi\left(  r\right)  +\frac{\zeta\left(  2\right)  }{2}%
+\frac{\left(  H_{r-1}\right)  ^{2}+H_{r-1}^{\left(  2\right)  }}%
{2}\nonumber\\
&  +\sum_{j=1}^{r-1}%
\genfrac{[}{]}{0pt}{}{r}{j}%
\left(  \widetilde{\zeta}_{H}\left(  r+1-j,r-1\right)  -H_{r-1}\zeta\left(
r+1-j\right)  \right)  . \label{teoghr2}%
\end{align}
On the other hand, it is known that the values $\widetilde{\zeta}_{H}\left(
p,r-1\right)  $ can be written as \cite[Theorem 2.1]{XL}
\begin{align}
\widetilde{\zeta}_{H}\left(  p,r-1\right)   &  =\frac{1}{2}\left(  p+2\right)
\zeta\left(  p+1\right)  -\frac{1}{2}\sum_{v=1}^{p-2}\zeta\left(  p-v\right)
\zeta\left(  v+1\right) \nonumber\\
&  -\sum_{v=1}^{p-1}\left(  -1\right)  ^{v}\zeta\left(  p+1-v\right)
H_{r-1}^{\left(  v\right)  }-\left(  -1\right)  ^{p}\sum_{v=1}^{r}\frac{H_{v}%
}{v^{p}} \label{14}%
\end{align}
for $p\in\mathbb{N}\backslash\left\{  1\right\}  .$ Combining (\ref{teoghr2})
and (\ref{14}) we accomplish that the constants $\gamma_{h^{\left(  r\right)
}}$ can be written in terms of zeta values, Euler-Mascheroni constant and
harmonic numbers.

Therefore it can be seen from (\ref{int-ie}), (\ref{teoghr2}) and (\ref{14})
that integrals involving $E_{p}^{0}\left(  t\right)  $ can be written in terms
of some special constants which is the assertion of Theorem \ref{teo-exp}.
\end{proof}

\begin{remark}
1) It follows from (\ref{teoghr2}) and (\ref{14}) that the constants
$\gamma_{h^{\left(  r\right)  }}$ can be explicitly written as%
\begin{align*}
\Gamma\left(  r\right)  \gamma_{h^{\left(  r\right)  }}  &  =-\frac{\gamma
^{2}}{2}-\gamma\psi\left(  r\right)  +\frac{\zeta\left(  2\right)  }{2}%
+\frac{\left(  H_{r-1}\right)  ^{2}+H_{r-1}^{\left(  2\right)  }}{2}\\
&  +\frac{1}{2}\sum_{j=1}^{r-1}%
\genfrac{[}{]}{0pt}{}{r}{j}%
\left(  r+3-j\right)  \zeta\left(  r+2-j\right)  -\sum_{v=1}^{r}H_{v}%
\sum_{j=1}^{r-1}%
\genfrac{[}{]}{0pt}{}{r}{j}%
\frac{\left(  -1\right)  ^{r+1-j}}{v^{r+1-j}}\\
&  +\sum_{v=2}^{r-1}\left(  \left(  -1\right)  ^{v-1}H_{r-1}^{\left(
v\right)  }-\frac{1}{2}\zeta\left(  v\right)  \right)  \sum_{j=v}^{r-1}%
\genfrac{[}{]}{0pt}{}{r}{r-j}%
\zeta\left(  2+j-v\right)  .
\end{align*}
In particular, when $r=1$ it reduces to
\[
\gamma_{h^{\left(  1\right)  }}=\frac{\gamma^{2}}{2}+\frac{\zeta\left(
2\right)  }{2}%
\]
and coincides with the constant $\gamma_{H}\left(  0\right)  $ recorded in
\cite[Eq. (6)]{CC}.

2) Equation (\ref{int-ie}) can be used recursively to find exact formulas for
the integrals mentioned in Theorem \ref{teo-exp}. For instance, we have
\begin{align*}
\int_{0}^{\infty}\left(  \frac{1}{t}-\frac{1}{1-e^{-t}}\right)  E_{1}%
^{0}\left(  t\right)  dt  &  =\gamma_{h^{\left(  1\right)  }}-\frac{1}%
{2}\gamma^{2}-\gamma\left(  1\right)  -\sigma_{1}+\frac{1}{2}\zeta\left(
2\right)  -1,\\
\int_{0}^{\infty}\left(  \frac{1}{t}-\frac{1}{1-e^{-t}}\right)  E_{2}%
^{0}\left(  t\right)  dt  &  =\gamma_{h^{\left(  2\right)  }}-\gamma
_{h^{\left(  1\right)  }}-2\zeta\left(  3\right)  -\zeta^{\prime}\left(
2\right)  +\sigma_{2}-\frac{3}{2}%
\end{align*}
and%
\[
\int_{0}^{\infty}\left(  \frac{1}{t}-\frac{1}{1-e^{-t}}\right)  E_{3}%
^{0}\left(  t\right)  dt =\gamma_{h^{\left(  3\right)  }}-\frac{3}{2}%
\gamma_{h^{\left(  2\right)  }}+\gamma_{h^{\left(  1\right)  }}-\frac{5}%
{4}\gamma-\sigma_{3}+\zeta^{\prime}\left(  3\right)  -\frac{\pi^{4}}{72}%
+\frac{3}{8}\pi^{2}-\frac{7}{12} .
\]

\end{remark}

\end{document}